\documentclass[reqno,centertags, 12pt]{amsart}
\usepackage{amsmath,amsthm,amscd,amssymb}
\usepackage{latexsym}
\usepackage{eucal}
\newcommand{\bbR}{{\mathbb{R}}}

\newcommand{\calH}{{\mathcal H}}
\newcommand{\calI}{{\mathcal I}}

\newcommand{\dott}{\,\cdot\,}

\newcommand{\lb}{\label}
\newcommand{\f}{\frac}

\newcommand{\ol}{\overline}

\newcommand{\tr}{\text{\rm{Tr}}}

\newcommand{\loc}{\text{\rm{loc}}}

\newcommand{\ess}{\text{\rm{ess}}}

\newcommand{\bi}{\bibitem}

\newcommand{\beq}{\begin{equation}}
\newcommand{\eeq}{\end{equation}}
\newcommand{\ba}{\begin{align}}
\newcommand{\ea}{\end{align}}

\newcounter{smalllist}

\allowdisplaybreaks
\numberwithin{equation}{section}

\newtheorem{theorem}{Theorem}[section]

\newtheorem*{t1}{Theorem 1}
\newtheorem*{t2}{Theorem 2}
\newtheorem*{t3}{Theorem 3}

\newtheorem*{p2.1}{Proposition 2.1}
\newtheorem{proposition}[theorem]{Proposition}

\theoremstyle{definition}

\theoremstyle{remark}
\newtheorem*{remark}{Remark}
\newtheorem*{remarks}{Remarks}

\newcommand{\abs}[1]{\lvert#1\rvert}
\newcommand{\jap}[1]{\langle #1 \rangle}
\newcommand{\norm}[1]{\lVert#1\rVert}

\begin{document}
\title[Schr\"odinger Operators with Purely Discrete Spectrum]
{Schr\"odinger Operators with\\Purely Discrete Spectrum}
\author[B.~Simon]{Barry Simon}

\dedicatory{Dedicated to A.~Ya.~Povzner}

\thanks{Mathematics 253-37, California Institute of Technology, Pasadena, CA 91125.
E-mail: bsimon@caltech.edu. Supported in part by NSF grant DMS-0652919 and by Grant No.\ 2006483
from the United States-Israel Binational Science Foundation (BSF), Jerusalem, Israel}
\thanks{To appear in the journal, {\it Methods of Functional Analysis and Topology}, volume in
memory of A.~Ya.~Povzner}

\date{October 13, 2008}
\keywords{compact resolvent, Schr\"odinger operators}
\subjclass[2000]{47B07, 35Q40, 47N50}

\begin{abstract} We prove $-\Delta +V$ has purely discrete spectrum if $V\geq 0$ and, for all $M$,
$\abs{\{x\mid V(x) <M\}}<\infty$ and various extensions.
\end{abstract}

\maketitle

\section{Introduction} \lb{s1}

Our main goal in this note is to explore one aspect of the study of Schr\"odinger operators
\begin{equation} \lb{1.1}
H=-\Delta + V
\end{equation}
which we'll suppose have $V$'s which are nonnegative and in $L_\loc^1(\bbR^\nu)$, in which case
(see, e.g., Simon \cite{S98}) $H$ can be defined as a form sum. We're interested here in criteria
under which $H$ has purely discrete spectrum, that is, $\sigma_\ess (H)$ is empty. This is
well known to be equivalent to proving $(H+1)^{-1}$ or $e^{-sH}$ for any (and so all) $s>0$ is
compact (see \cite[Thm.~XIII.16]{RS4}). One of the most celebrated elementary results on Schr\"odinger
operators is that this is true if
\begin{equation} \lb{1.2}
\lim_{\abs{x}\to\infty} V(x) =\infty
\end{equation}

But \eqref{1.2} is not necessary. Simple examples where \eqref{1.2} fails but $H$ still has compact
resolvent were noted first by Rellich \cite{Rel}---one of the most celebrated examples is in $\nu=2$,
$x=(x_1,x_2)$, and
\begin{equation} \lb{1.3}
V(x_1,x_2) =x_1^2 x_2^2
\end{equation}
where \eqref{1.2} fails in a neighborhood of the axes. For proof of this and discussions of eigenvalue
asymptotics, see \cite{Rob,S159,S158,Sol,Tam}.

There are known necessary and sufficient conditions on $V$ for discrete spectrum in terms of capacities
of certain sets (see, e.g., Maz'ya \cite{Maz}), but the criteria are not always so easy to check. Thus,
I was struck by the following simple and elegant theorem:

\begin{t1} Define
\begin{equation} \lb{1.4}
\Omega_M (V) =\{x\mid 0\leq V(x) <M\}
\end{equation}
If {\rm{(}}with $\abs{\dott}$ Lebesgue measure{\rm{)}}
\begin{equation} \lb{1.5}
\abs{\Omega_M(V)}<\infty
\end{equation}
for all $M$\!, then $H$ has purely discrete spectrum.
\end{t1}

I learned of this result from Wang--Wu \cite{WW}, but there is much related work. I found an elementary
proof of Theorem~1 and decided to write it up as a suitable tribute and appreciation of A.~Ya.~Povzner,
whose work on continuum eigenfunction expansions for Schr\"odinger operators in scattering situation
\cite{Pov55} was seminal and inspired  me as a graduate student forty years ago!

The proof has a natural abstraction:

\begin{t2} Let $\mu$ be a measure on a locally compact space, $X$ with $L^2 (X,d\mu)$ separable.
Let $L_0$ be a selfadjoint operator on $L^2(X,d\mu)$ so that its semigroup is ultracontractive
{\rm{(}}\cite{S173}{\rm{)}}: For some $s>0$, $e^{-sL_0}$ maps $L^2$ to $L^\infty (X,d\mu)$. Suppose
$V$ is a nonnegative multiplication operator so that
\begin{equation} \lb{1.6}
\mu(\{x\mid 0\leq V(x) <M\}) <\infty
\end{equation}
for all $M$\!. Then $L=L_0+V$ has purely discrete spectrum.
\end{t2}

\begin{remark} By $L_0+V$, we mean the operator obtained by applying the monotone convergence theorem for
forms (see, e.g., \cite{S82,S81}) to $L_0 + \min(V(x),k)$ as $k\to\infty$.
\end{remark}

The reader may have noticed that \eqref{1.3} does not obey Theorem~1 (but, e.g., $V(x_1, x_2)=x_1^2 x_2^4
+ x_1^4 x_2^2$ does). But out proof can be modified to a result that does include \eqref{1.3}. Given a
set $\Omega$ in $\bbR^\nu$, define for any $x$ and any $\ell >0$,
\begin{equation} \lb{1.7}
\omega_x^\ell (\Omega) = \abs{\Omega\cap \{y\mid \abs{y-x}\leq\ell\}}
\end{equation}
For example, for \eqref{1.3}, for $x\in\Omega_M$,
\begin{equation} \lb{1.8}
\omega_x^\ell (\Omega_M) \leq \f{C_\ell}{\abs{x}+1}
\end{equation}
We will say a set $\Omega$ is $r$-polynomially thin if
\begin{equation} \lb{1.9}
\int_{x\in\Omega} \omega_x^\ell (\Omega)^r\, d^\nu x <\infty
\end{equation}
for all $\ell$. For the example in \eqref{1.3}, $\Omega_M$ is $r$-polynomially thin for any $M$ and any
$r>0$. We'll prove

\begin{t3} Let $V$ be a nonnegative potential so that for any $M$\!, there is an $r>0$ so that $\Omega_M$
is $r$-polynomially thin. Then $H$ has purely discrete spectrum.
\end{t3}

As mentioned, this covers the example in \eqref{1.3}. It is not hard to see that if $P(x)$ is any
polynomial in $x_1, \dots, x_\nu$ so that for no $v\in\bbR^\nu$ is $\vec{v}\cdot\vec{\nabla} P\equiv 0$
(i.e., $P$ isn't a function of fewer than $\nu$ linear variables), then $V(x)=P(x)^2$ obeys the
hypotheses of Theorem~3.

In Section~\ref{s2}, we'll present a simple compactness criterion on which all theorems rely. In Section~\ref{s3},
we'll prove Theorems~1 and 2. In Section~4, we'll prove Theorem~3.

\medskip
It is a pleasure to thank Peter Stollmann for useful correspondence and Ehud de~Shalit for the hospitality
of Hebrew University where some of the work presented here was done.

\section{Segal's Lemma} \lb{s2}

Segal \cite{Seg} proved the following result, sometimes called Segal's lemma:

\begin{proposition} \lb{P2.1} For $A,B$ positive selfadjoint operators,
\begin{equation} \lb{2.1}
\norm{e^{-(A+B)}} \leq \norm{e^{-A} e^{-B}}
\end{equation}
\end{proposition}

\begin{remarks} 1. $A+B$ can always be defined as a closed quadratic form on $Q(A)\cap Q(B)$. That defines
$e^{-(A+B)}$ on $\ol{Q(A)\cap Q(B)}$ and we set it to $0$ on the orthogonal complement. Since the Trotter
product formula is known in this generality (see Kato \cite{Kato}), \eqref{2.1} holds in that generality.

\smallskip
2. Since $\norm{C^* C}=\norm{C}^2$, $\norm{e^{-A/2} e^{-B/2}}^2 = \norm{e^{-B/2} e^{-A} e^{-B/2}}$, and
since $\norm{e^{-(A+B)/2}}^2 = \norm{e^{-(A+B)}}$, \eqref{2.1} is equivalent to
\begin{equation} \lb{2.2}
\norm{e^{-A+B}} \leq \norm{e^{-B/2} e^{-A} e^{-B/2}}
\end{equation}
which is the way Segal \cite{Seg} stated it.

\smallskip
3. Somewhat earlier, Golden \cite{Gold} and Thompson \cite{Tho65} proved
\begin{equation} \lb{2.3}
\tr (e^{-(A+B)}) \leq \tr (e^{-A} e^{-B})
\end{equation}
and Thompson \cite{Tho71} later extended this to any symmetrically normed operator ideal.
\end{remarks}

\begin{proof} There are many; see, for example, Simon \cite{STI,Sr44}. Here is the simplest,
due to Deift \cite{Deift76,Deift78}: If $\sigma$ is the spectrum of an operator
\begin{equation} \lb{2.4}
\sigma(C\!D)\setminus\{0\}=\sigma(DC)\setminus\{0\}
\end{equation}
so with $\sigma_r$ the spectral radius,
\begin{equation} \lb{2.5}
\sigma_r(C\! D) = \sigma_r(DC)\leq \norm{DC}
\end{equation}
If $C\! D$ is selfadjoint, $\sigma_r (C\! D)=\norm{C\! D}$, so
\begin{equation} \lb{2.6}
C\! D\text{ selfadjoint} \Rightarrow \norm{C\! D}\leq \norm{DC}
\end{equation}

Thus,
\begin{equation} \lb{2.7}
\norm{e^{-A/2} e^{-B/2}}^2 = \norm{e^{-B/2} e^{-A} e^{-B/2}} \leq \norm{e^{-A} e^{-B}}
\end{equation}
By induction,
\begin{equation} \lb{2.8}
\norm{(e^{-A/2^n} e^{-B/2^n})^{2^n}} \leq \norm{e^{-A/2^n} e^{-B/2^n}}^{2n}
\leq \norm{e^{-A} e^{-B}}
\end{equation}
Take $n\to\infty$ and use the Trotter product formula to get \eqref{2.1}.
\end{proof}

In \cite{STI}, I noted that this implies for any symmetrically normed trace ideal, $\calI_\Phi$, that
\begin{equation} \lb{2.9}
e^{-A/2} e^{-B} e^{-A/2} \in \calI_\Phi \Rightarrow e^{-(A+B)}\in\calI_\Phi
\end{equation}
I explicitly excluded the case $\calI_\Phi = \calI_\infty$ (the compact operators) because the
argument there doesn't show that, but it is true---and the key to this paper!

Since $C\in\calI_\infty \Leftrightarrow C^* C\in\calI_\infty$ and $e^{-(A+B)}\in\calI_\infty$ if and
only if $e^{-\f12 (A+B)}\in\calI_\infty$, it doesn't matter if we use the symmetric form \eqref{2.2}
or the following asymmetric form which is more convenient in applications.

\begin{theorem}\lb{T2.2} Let $\calI_\infty$ be the ideal of compact operators on some Hilbert space,
$\calH$. Let $A,B$ be nonnegative selfadjoint operators. Then
\begin{equation} \lb{2.10}
e^{-A} e^{-B} \in\calI_\infty \Rightarrow e^{-(A+B)}\in\calI_\infty
\end{equation}
\end{theorem}

\begin{proof} For any bounded operator, $C$, define $\mu_n(C)$ by
\begin{equation} \lb{2.11}
\mu_n(C) = \min_{\psi_1 \dots \psi_{n-1}} \,
\sup_{\substack{ \norm{\varphi}=1 \\ \varphi\perp\psi_1, \dots, \psi_{n-1}}} \norm{C\varphi}
\end{equation}
By the min-max principle (see \cite[Sect.~XIII.1]{RS4}),
\begin{equation} \lb{2.12}
\lim_{n\to\infty}\, \mu_n(C) =\sup (\sigma_\ess (\abs{C}))
\end{equation}
and $\mu_n(C)$ are the singular values if $C\in\calI_\infty$. In particular,
\begin{equation} \lb{2.13}
C\in\calI_\infty \Leftrightarrow \lim_{n\to\infty}\, \mu_n(C)=0
\end{equation}

Let $\wedge^\ell(\calH)$ be the antisymmetric tensor product (see \cite[Sects.~II.4, VIII.10]{RS1},
\cite[Sect.~XIII.17]{RS4}, and \cite[Sect.~1.5]{STI}). As usual (see \cite[eqn.~(1.14)]{STI}),
\begin{equation} \lb{2.14}
\norm{\wedge^m (C)} = \prod_{j=1}^m \mu_j(C)
\end{equation}
Since $\mu_1\geq\mu_2 \geq \cdots \geq 0$, we have
\begin{equation} \lb{2.15}
\lim_{n\to\infty}\, \mu_n(C) =\lim_{n\to\infty}\, (\mu_1(C) \dots \mu_n(C))^{1/n}
\end{equation}

\eqref{2.13}--\eqref{2.15} imply
\begin{equation} \lb{2.16}
C\in\calI_\infty \Leftrightarrow \lim_{n\to\infty}\, \norm{\wedge^n (C)}^{1/n} =0
\end{equation}

As usual, there is a selfadjoint operator, $d\wedge^n(A)$ on $\wedge^n(\calH)$ so
\begin{equation} \lb{2.17}
\wedge^n (e^{-tA})=e^{-t\, d\wedge^n(A)}
\end{equation}
so Segal's lemma implies that
\begin{align}
\norm{\wedge^n (e^{-(A+B)})} &\leq \norm{\wedge^n (e^{-A}) \wedge^n (e^{-B})} \notag \\
&= \norm{\wedge^n (e^{-A} e^{-B})} \lb{2.18}
\end{align}

Thus,
\begin{equation} \lb{2.19}
\lim_{n\to\infty}\, \norm{\wedge^n (e^{-(A+B)})}^{1/n} \leq \lim_{n\to\infty}\,
\norm{\wedge^n (e^{-A} e^{-B})}^{1/n}
\end{equation}

By \eqref{2.16}, we obtain \eqref{2.10}.
\end{proof}

\section{Proofs of Theorems~1 and 2} \lb{s3}

\begin{proof}[Proof of Theorem~1] By Theorem~\ref{T2.2}, we need only show $C=e^\Delta e^{-V}$ is
compact. Write
\begin{equation} \lb{3.1}
C=C_m +D_m
\end{equation}
where
\begin{equation} \lb{3.2}
C_m = C\chi_{\Omega_m} \qquad D_m =C\chi_{\Omega_m^c}
\end{equation}
with $\chi_S$ the operator of multiplication by the characteristic function of a set $S\subset\bbR^\nu$.
\[
\norm{e^{-V} \chi_{\Omega_m^c}}_\infty \leq e^{-m}
\]
and $\norm{e^\Delta}=1$, so
\begin{equation} \lb{3.3}
\norm{D_m} \leq e^{-m}
\end{equation}
and thus,
\begin{equation} \lb{3.4}
\lim_{m\to\infty}\, \norm{C-C_m} =0
\end{equation}

If we show each $C_m$ is compact, we are done. We know $e^\Delta$ has integral kernel $f(x-y)$ with
$f$ a Gaussian, so in $L^2$. Clearly, since $V$ is positive, $C_m$ has an integral kernel $C_m(x,y)$
dominated by
\begin{equation} \lb{3.5}
\abs{C_m(x,y)} \leq f(x-y) \chi_{\Omega_m}(y)
\end{equation}
Thus,
\[
\int \abs{C_m(x,y)}^2\, d^\nu x d^\nu y \leq \norm{f}^2_{L^2(\bbR^\nu)}
\norm{\chi_{\Omega_m}}_{L^2(\bbR_\nu)} <\infty
\]
since $\abs{\Omega_m} <\infty$. Thus, $C_m$ is Hilbert--Schmidt, so compact.
\end{proof}

\begin{proof}[Proof of Theorem~2] We can follow the proof of Theorem~1. It suffices to prove that
$e^{-sL_0} e^{-sV}$ is compact, and so, that $e^{-sL_0} \chi_{\Omega_m}$ is Hilbert--Schmidt.

That $e^{-sL_0}$ maps $L^2$ to $L^\infty$ implies, by the Dunford--Pettis theorem (see
\cite[Thm.~46.1]{Tre}), that there is, for each $x\in X$\!, a function $f_x(\dott)\in L^2 (X,d\mu)$
with
\begin{equation} \lb{3.6}
(e^{-sL_0} g)(x) =\jap{f_x,g}
\end{equation}
and
\begin{equation} \lb{3.7}
\sup_x \, \norm{f_x}_{L^2} =\norm{e^{-sL_0}}_{L^2\to L^\infty} \equiv C <\infty
\end{equation}

Thus, $e^{-sL_0}$ has an integral kernel $K(x,y)$ with
\begin{equation} \lb{3.8}
\sup_x \int \abs{K(x,y)}^2\, d\mu(y) = C<\infty
\end{equation}
(for $K(x,y)=f_x(y)$). But $e^{-sL_0}$ is selfadjoint, so its kernel is complex symmetric, so
\begin{equation} \lb{3.9}
\sup_y \int \abs{K(x,y)}^2\, d\mu(x) =C <\infty
\end{equation}
Thus,
\begin{equation} \lb{3.10}
\int \abs{K(x,y) \chi_{\Omega_m}(y)}^2 \, d\mu(x) d\mu(y)\leq C\mu (\Omega_m) <\infty
\end{equation}
and $e^{-sL_0}\chi_{\Omega_m}$ is Hilbert--Schmidt.
\end{proof}

\section{Proof of Theorem~3} \lb{s4}

As with the proof of Theorem~1, it suffices to prove that for each $M$\!, $e^\Delta \chi_{\Omega_M}$
is compact. $e^\Delta$ is convolution with an $L^1$ function, $f$. Let $Q_R$ be the characteristic
function of $\{x\mid\abs{x} <R\}$. Let $F_R$ be convolution with $fQ_R$. Then
\begin{equation} \lb{4.1}
\norm{e^\Delta -F_R} \leq \norm{f(1-Q_R)}_1 \to 0
\end{equation}
as $R\to\infty$, so
\begin{equation} \lb{4.2}
\norm{e^\Delta \chi_{\Omega_M} -F_R \chi_{\Omega_M}} \to 0
\end{equation}
and it suffices to prove for each $R,M$\!,
\begin{equation} \lb{4.3}
C_{M,R} =F_R \chi_{\Omega_M}
\end{equation}
is compact. Clearly, this works if we show for some $k$, $(C_{M,R}^* C_{M,R})^k$ is Hilbert--Schmidt.

Let $D$ be the operator with integral kernel
\begin{equation} \lb{4.4}
D(x,y) =\chi_{\Omega_M}(x) Q_{2R} (x-y) \chi_{\Omega_M}(y)
\end{equation}
Since $f$ is bounded, it is easy to see that
\begin{equation} \lb{4.5}
(C_{M,R}^* C_{M,R})(x,y) \leq cD(x,y)
\end{equation}
for some constant $c$, so it suffices to show $D^k$ is Hilbert--Schmidt.

$D^k$ has integral kernel
\begin{equation} \lb{4.6}
D^k (x,y) = \int D(x,x_1) D(x_1,x_2) \dots D(x_{k-1},y)\, dx_1 \dots, dx_{k-1}
\end{equation}

Fix $y$. This integral is zero unless $\abs{x-x_1} <2R, \dots \abs{x_{k-1}-y} <2R$, so, in particular,
unless $\abs{x-y}\leq 2kR$. Moreover, the integrand can certainly be restricted to the regions
$\abs{x_j -y}\leq 2kR$. Thus,
\begin{align}
D^k(x,y) &\leq Q_{2kR}(x-y) \biggl( \int_{\abs{x_j-y}\leq 2kR}\, \prod_{j=1}^{k-1} \chi_{\Omega_M} (x_j) \,
dx_1 \dots dx_{k-1}\biggr) \chi_{\Omega_m}(y) \lb{4.7} \\
&= Q_{2kR} (x-y) (\omega_y^{2kR}(\Omega_M)^{k-1}) \chi_{\Omega_M}(y) \lb{4.8}
\end{align}
by the definition of $\omega_x^\ell$ in \eqref{1.7}.

Thus,
\[
\int \abs{D^k(x,y)}^2\, d^\nu x d^\nu y \leq C(kR)^\nu \int_{x\in\Omega}
[\omega_x^{2kR} (\Omega_M)]^{2k-2}\, d^\nu x
\]
so if $2k-2 >r$ and \eqref{1.9} holds, $D^k$ is Hilbert--Schmidt.
\hfill \qed

\bigskip

\end{document}